\documentclass[oneside,a4paper,leqno,11pt]{article}

\usepackage[english]{babel}
\usepackage{amssymb}
\usepackage{dsfont}
\usepackage{amsmath}
\usepackage{latexsym, enumerate, amsfonts, amsthm, amsopn, amstext, amscd}
\usepackage{fancyhdr}
\fancypagestyle{plain}
\def\R{\mathds{R}}
\def\N{\mathds{N}}
\def\I{\mathds{I}}
\def\Z{\mathds{Z}}

\newtheorem{theorem}{Theorem}[section]

\newtheorem{lemma}[theorem]{Lemma}

\newtheorem{remark}{Remark}[section]

\def\endproof{\hfill $\square$
\endtrivlist}

\begin{document}
\parskip5pt
 \author{A. C. Rosa, M. E. Nogueira \footnote{Adress: Departamento de Matem\'{a}tica, Apartado 3008, EC Santa Cruz, $3001-501$ Coimbra. E-mail: cristina@mat.uc.pt, memn@mat.uc.pt.}}

\title{Nonparametric estimation of a regression function using the gamma kernel method in ergodic processes \footnote{The work was supported by the Department of Mathematics of the University of Coimbra.}}
\maketitle

\vspace*{1cm}

\noindent{\bf Abstract.} In this paper we consider the nonparametric estimation of density and regression functions with non-negative support using a gamma kernel procedure introduced by Chen (\cite{Chen:00}). Strong uniform consistency and asymptotic normality of the corresponding estimators are established under a general ergodic assumption on the data generation process. Our results generalize those of Shi and Song (\cite{ShiSong:15}), obtained in the classic i.i.d. framework, and the works of Bouezmarni and Rombouts \cite{BouRomb:08,BouRomb:10b} and Gospodinov and Hirukawa \cite{GospHiru:07} for mixing time series.
\vspace*{0.25cm}

\vspace*{0.25cm}


\section{Introduction}
\label{Sec1}

As it is well known, a major drawback of the standard kernel method for nonparametric curve estimation concerns the presence of the so-called bounded effects. Bounded effects occur when the support of the underlying variables is a subinterval of the real line and the estimates are based on a symmetric kernel, leading to an increase of the bias near the boundary of the support.  Since the pioneering works of Gasser and M\"{u}ller (\cite{GassMuller:79}), Rice (\cite{Rice:84}), Schuster (\cite{Schuster:85}) and Gasser et al. (\cite{GassMullerMammit:85}), several approches to overcome this problem have been investigated (for an overview of the main correction techniques, the reader is referred to Simonoff (\cite{Simonoff:98}), Karunamuni and Alberts (\cite{KarunAlberts:05}) and Dai and Sperlich (\cite{DaiSperl:10})). Among the existing proposals, the boundary kernel method has shown to be one of the most popular. The general idea behind this method is to modify the kernel's form near the endpoints of the support, either by using adaptive kernels in the boundary region and a fixed symmetric kernel in the interior or by considering asymmetric kernels, whose shape and scale parameters change in accordance with the position of the target point, allowing to adjust the local smoothness of the estimate in a natural way.

Asymmetric kernels, namely beta and gamma kernels, were introduced by Brown and Chen (\cite{BrChen:99}) and Chen (\cite{Chen:99},\cite{Chen:00}) to estimate densities supported in $[0,1]$ and $[0,+\infty[$, respectively. Apart from having the same support as the curve under consideration, they present other appealing features such as achieving the optimal convergence rate for the mean integrated square error of classical kernels and showing good finite sample performance.

Regarding gamma kernels, which are the goal of our study, Chen's proposal and its refinements remained a topic of interest for researchers (c.f. Geenens and Wang (\cite{GeenWang:16}) for an review on the subject and Malec and Schienle (\cite{MaleSchi:14}) and Hirukawa and Sakuda (\cite{HiruSak:14}) for a recent simulation studies), although other types of asymmetric kernels have been suggested in the last decades (e.g. the inverse gaussian and the reciprocal inverse gaussian kernels of Scaillet (\cite{Scaillet:04})). However, as pointed out by Koul and Song (\cite{KoulSong:13}), most existing results are devoted to density estimation and address essentially asymptotic bias, variance and mean square error derivations in the i.i.d. setting. In the last few years, there has been increasing attention given to consistency and limiting distributions of both density and regression estimators in this context (c.f. Bouezmarni and Rombouts (\cite{BouRomb:10a}), Bouezmarni et al. (\cite{BouGhMes:11}), Shi and Song (\cite{ShiSong:15}), Koul and Song (\cite{KoulSong:13})) as well as their natural extensions in stationary time series context (c.f., for instance, Bouezmarni and Rombouts (\cite{BouRomb:08},\cite{BouRomb:10b}), Markovich (\cite{Markovich:15}), in the class of mixing processes, and Chaubey et al. (\cite{ChauLaiSen:12}) in a larger class).

The last authors, arguing that the traditional mixing hypotheses imposed on the observation process are not satisfied in many cases (several examples of ergodic and non-mixing processes may also be found in Bouzebda and Didi (\cite{Bouz:15})), worked under the general dependence condition of ergodicity introduced by Gy\"{o}rfi (\cite{Gyorfi:81}) and considered by Delecroix et al. (\cite{DelNogRosa:92}), Delecroix and Rosa (\cite{DelRosa:96}), Yakowitz et al. (\cite{YakGyorfKiefferMorvai:99}) and La\"{\i}b and Ould-Sa\"{\i}d (\cite{LaibOuldSaid:00}). In fact, this condition gained a renewed interest after the paper of La\"{\i}b and Louani (\cite{LaibLouani:10}), giving rise to some new consistency results with convergence rate for nonparametric curve estimation  (cf. La\"{\i}b and Louani (\cite{LaibLouani:10},\cite{LaibLouani:11}), Chaouch and Khardani (\cite{ChaKhar:15}), Bouzebda and Didi (\cite{Bouz:15}), Benziadi et al. (\cite{Benz:16}), Ling and Liu (\cite{LingLiu:16}), Ling et al. (\cite{LingLiuVieu:16})).

Following the works of Chaubey et al. (\cite{ChauLaiSen:12}) and Shi and Song (\cite{ShiSong:15}), we prove, in the present paper, the uniform consistency and the asymptotic normality of density and regression estimators based on the gamma kernel proposed by Chen (\cite{Chen:00}), in the framework of discrete time ergodic processes. With respect to the central limit theorem, we remark that, under mild conditions imposed on the bandwidth, the convergence rates and the asymptotic variances obtained in our work agree with those of Shi and Song (\cite{ShiSong:15}) considering the i.i.d. setup.

The paper is organized as follows: section \ref{Sec2} introduces the estimators as well as the general notations and assumptions on the observation process; section \ref{Sec3} provides the main convergence results and a few commentaries concerning their hypotheses; the proofs of the propositions and some auxiliary lemmas are presented in section \ref{sec4}.

\section{Assumptions and notations}
\label{Sec2}

Let $\left\{(X_t,Y_t),t \in \Z\right\}$ be a $\left(\R_0^{+}\right)^2$-valued
 stochastic process on the proba\-bi\-lity space $\left(\Omega, \mathcal{A}, P\right)$ which is assumed to be strictly
stationary and ergodic, with absolutely continuous margin distributions.
The density function of $X_t$ will be denoted by $f$.

For each $x \in \R_0^+$ such that $f(x)>0$, $R(x) = E \left(\Phi(Y_1)/ X_1=x\right)$ stands for the conditional expectation of $\Phi(Y_1)$ given $X_1=x$, where $\Phi$ is a known measurable function of $\R_0^+$ into $\R$ such that $E \left(\left|\Phi(Y_1)\right|\right)<+\infty$.

Based on a sample $\{(X_t,Y_t)\}_{_{t=1}}^{n}$, our goal is to study some asymptotic pro\-per\-ties of the following estimator of $R(x)$ $$R_n(x)= \frac{\,\overset{n}{\underset{t=1}{\sum }}\ \Phi(Y_t)\, K_{_{\alpha(n,x),\beta(n)}}\!(X_t)\ }{\overset{n}{\underset{t=1}{\sum }}\,K_{_{\alpha(n,x),\beta(n)}}\!(X_t)},$$
$K_{_{\alpha(n,x),\beta(n)}}$ being the density function of the gamma distribution with shape and scale parameters $\alpha(n,x)= \frac{x}{\,h_n\,} + 1$ and $\beta(n)= h_n$, respectively given by $$\textstyle K_{_{\alpha(n,x),\beta(n)}}\!(y) = \frac{1}{\,\Gamma(\alpha(n,x))\,\beta(n)^{\alpha(n,x)\,}}\  y^{\alpha(n,x)-1}\ e^{-\,\frac{y}{\,\beta(n)\,}}\ \I_{[0,+\infty[}(y)\,,\ \ \ y \in \R.$$
As usual, $(h_n)_{_{n\in\N}}$ is the bandwidth sequence, i.e. $h_n \in \R^+$,\ $n \in \N$, $\underset{n\rightarrow+\infty}{\lim}h_n = 0$, and we adopt the convention that $\frac{\,y\,}{0}=0$, for all $y\in\R$.

 \vspace*{0.1cm}

In the sequel, we will consider the $\sigma$-fields \vspace*{0.25cm}

\centerline{$\mathcal{F}_t = \sigma\left\{(X_s, Y_s), \,s \leq t \right\},\ \ \mathcal{G}_t = \sigma\left\{X_{t+1}, (X_s, Y_s), s \leq t\right\}$,\ \ $t \in \Z$,\vspace*{0.2cm}}


\noindent and we will denote by $C_0(\R)$ the space of continuous functions on $\R$ tending to zero at infinity equipped with the sup-norm, $\|.\|_0$, and by $\|.\|_2$ the norm in $L^2\left(\Omega, \mathcal{A}, P\right)$.

For easy reference, the general assumptions needed to derive the announced results are gathered thereafter. \vspace*{-0.15cm}
\begin{enumerate}
\item[(H1)] For all $t \in \Z$, the conditional density of $X_t$ given $\mathcal{F}_{t-1}$, $f^{\mathcal{F}_{t-1}}$, exists; moreover, $f^{\mathcal{F}_{t-1}} \in C_0(\R)$ and $f \in C_0(\R)$.

\item[(H2)] $\left\|\,\frac{1}{\,n\,}\overset{n}{\underset{t=1}{\sum }}\,f^{\mathcal{F}_{t-1}}-f\,\right\|_0 \overset{a.s.}{\overrightarrow{\underset{n\rightarrow + \infty}}}\ 0$.

\item[(H3)] $R$ is a continuous and bounded function on $\R_0^+$.

\item[(H4)] $E\!\left( \Phi(Y_t) / \mathcal{G}_{t-1}\right)= E\!\left( \Phi(Y_t) / X_t\right)= R( X_t),\ \ t \in \Z.$
 \end{enumerate}

\section{Main results}
\label{Sec3}

\subsection{Strong uniform consistency of $D_n$ and $R_n$}

In order to establish the uniform convergence of $R_n$ on $\Delta = [a,b]$,\ $a,b\in\R^+$,\ $a<b$, we need a preliminary result concerning the behaviour of the gamma kernel estimator of $f$, i.e.  \vspace*{-0.15cm}
\begin{equation}\label{defDn}
\textstyle D_n(x) = \frac{1}{\,n\,} \overset{n}{\underset{t=1}{\sum }}\,K_{_{\alpha(n,x),\beta(n)}}\!(X_t).
\end{equation}

\begin{theorem}\label{convsupdens}
  If conditions (H1) and (H2) are satisfied and the sequence  $(h_n)_{_{n\in\N}}$ is such that \vspace*{-0.2cm}
$$ \underset{n \rightarrow + \infty}{\lim}\ \frac{\ n\,h_n\ }{\log n}\ =\ +\infty,\vspace*{-0.2cm}$$
\noindent then 
$$\underset{x\in \Delta}{\sup}\left|D_n(x) - f(x)\right|\ \overset{a.s.}{\overrightarrow{\underset{n\rightarrow + \infty}}}\ 0.$$
\end{theorem} \vspace*{0.1cm}


\begin{theorem}\label{convsupremo}
  In addition to conditions (H1) to (H4), suppose that\ \ $$E\!\left( \left|\Phi(Y_1)\right|^{\tau+1}\right) < + \infty,\ for\ some\ \tau > 0.$$
   If $I = {\underset {x\in \Delta}{\inf}}f(x)>0$\ \ and\ \ the sequence  $(h_n)_{_{n\in\N}}$ verifies \vspace{-0.25cm}
\begin{center}
$ n\sqrt{h_n\,}\ {\big{\uparrow}} +\!\infty,\ \ \ \ \ \exists\, \theta \in \left]0,\frac{\tau}{\,\tau+1\,}\right[\,:\ \ \ \displaystyle \underset{n \rightarrow + \infty}{\lim}\ \frac{\ n^{\theta}\,h_n\ }{\log n} = +\infty,$ \vspace{-0.25cm}
\end{center}
\noindent we have
$$\underset{x\in \Delta}{\sup}\left|R_n(x) - R(x)\right|\ \overset{a.s.}{\overrightarrow{\underset{n\rightarrow + \infty}}}\ 0.$$
\end{theorem}

\subsection{Asymptotic normality of $D_n$ and $R_n$}

Let us begin by presenting some addi\-tional assumptions. 
\begin{enumerate}
\item[(H5)] $E\!\left( \Phi^{\zeta+2}(Y_1)\right)$ exists, for some $\zeta>0$.\vspace*{0.1cm}
    \item[(H6)]
 \begin{enumerate}
 \item[(i)] $E\!\left( \Phi(Y_t)^2 / \mathcal{G}_{t-1}\right)= E\!\left( \Phi(Y_t)^2 / X_t\right)= W_2( X_t),\ \ t\in\Z .$\vspace*{0.1cm}
 \item[(ii)] $W_{\zeta + 2}( y)= E\!\left( \left|\Phi(Y_1)\right|^{\zeta+2}/ X_1 = y\right),\ \ y \in \R^+_0$, is a bounded function.\vspace*{0.1cm}
\item[(iii)] $\sigma^2( y)= V\!\left( \Phi(Y_1)/ X_1 = y\right),\ \ y \in \R^+_0$, is a continuous function.
 \end{enumerate}
\item[(H7)] The second order derivatives of $f$ and $R$ are continuous and bounded on $\R^+_0$.
\item[(H8)] $\underset{y \in \R^+_0}{\sup} \left\| \overset{n}{\underset{t=1}{\sum }}\,f^{\mathcal{F}_{t-1}}(y)-nf(y) \right\|^2_2 = O(n).$
\end{enumerate}

We are now in position to state the central limit theorems concerning the gamma kernel estimators of $f$ and $R$. \vspace*{0.25cm}

\begin{theorem}\label{TeorTLCdensity}
Let $x\in\R_0^+$ be such that $f(x)>0$. In addition to (H1), (H2), (H7) and (H8), suppose that the sequence $\left(\frac{1}{\,n\,}\overset{n}{\underset{t=1}{\sum }}\,\left(f^{\mathcal{F}_{t-1}}\right)^2\right)_{n \in \N}$ converges in $\,C_0(\R)$.
If \vspace*{-0.1cm}
\begin{enumerate}
  \item[a)] $\underset{n\rightarrow+\infty}{\lim} n\,\sqrt{h_n\,} = + \infty$\ \ and\ \ $\underset{n\rightarrow+\infty}{\lim} n\,\sqrt{h_n^5\,}=0 $, then
 $$\textstyle\sqrt{n\,\sqrt{h_n\,}\,}\ (D_n(x)-f(x))\ \overset{\mathcal{D}}{\overrightarrow{\underset{n\rightarrow + \infty}}}\ \ \mathcal{N}\left(0,  \frac{f(x)}{\ 2\,\sqrt{\pi\,x\,}\ }\right),\ \ \ x>0;$$
  \item[b)] $\underset{n\rightarrow+\infty}{\lim} n\,h_n = + \infty$\ \ and\ \ $\underset{n\rightarrow+\infty}{\lim} n\,h_n^3=0$, then \vspace*{-0.2cm} $$\textstyle \sqrt{n\,h_n\,}\ (D_n(0)-f(0))\ \overset{\mathcal{D}}{\overrightarrow{\underset{n\rightarrow + \infty}}}\ \ \mathcal{N}\left(0, \frac{\,f(0)\,}{\,2\,}\right).\vspace*{0.25cm}$$
\end{enumerate}
\end{theorem} \vspace*{0.1cm}

\begin{theorem} \label{asympNormal} Let $x\in\R_0^+$ be such that $f(x)>0$ and suppose that (H1) to (H8) hold. If \vspace*{-0.1cm}
\begin{enumerate}
  \item[a)] $\underset{n\rightarrow+\infty}{\lim} n\,\sqrt{h_n} = + \infty$\ \ and\ \ $\underset{n\rightarrow+\infty}{\lim} n\,\sqrt{h_n^5\,}=0 $, then \vspace*{-0.2cm}
 $$\textstyle \sqrt{n\,\sqrt{h_n\,}\,}\ (R_n(x)-R(x))\ \overset{\mathcal{D}}{\overrightarrow{\underset{n\rightarrow + \infty}}}\ \ \mathcal{N}\left(0,  \frac{\sigma^2(x)}{\,2\,\sqrt{\pi\,x\,}f(x)\,}\right),\ \ \ x>0;$$

  \item[b)] $\underset{n\rightarrow+\infty}{\lim} n\,h_n = + \infty$\ \ and\ \ $\underset{n\rightarrow+\infty}{\lim} n\,h_n^3=0$, then \vspace*{-0.2cm} $$\textstyle \sqrt{n\,h_n\,}\ (R_n(0)-R(0))\ \overset{\mathcal{D}}{\overrightarrow{\underset{n\rightarrow + \infty}}}\ \ \mathcal{N}\left(0, \frac{\sigma^2(0)}{\ 2\,f(0)\ }\right).\vspace*{0.25cm}$$
\end{enumerate}
\end{theorem}

The general conditions (H1) to (H8) as well as the hypotheses of Theorems \ref{convsupremo}, \ref{TeorTLCdensity} and \ref{asympNormal} will be discussed in the next section.

\subsection{Comments on the assumptions}
\label{subsec3}

We remark that assumption (H2) as well as the hypothesis concerning the convergence of $\,\left(\frac{1}{\,n\,}\overset{n}{\underset{t=1}{\sum }}\,\left(f^{\mathcal{F}_{t-1}}\right)^2\right)_{\!\!n \in \N}$ rely on the ergodic character of the data and became quite common in the general framework of ergodicity considered in the present paper (c.f. Delecroix et al. (\cite{DelNogRosa:92}), Delecroix and Rosa (\cite{DelRosa:96}), La\"{i}b and Ould-Sa\"{i}d (\cite{LaibOuldSaid:00}) and, more recently, condition (A2) (iii) of La\"{\i}b and Louani (\cite{LaibLouani:10},\cite{LaibLouani:11}), condition (A2) of  Chaubey  et al. (\cite{ChauLaiSen:12}), conditions (C1), (C2), (N1), (N5) of Bouzebda and Didi (\cite{Bouz:15}), condition (A3) (iii) of Ling and Liu (\cite{LingLiu:16}), condition (A2) 3. of Ling et al. (\cite{LingLiuVieu:16})).

Assumptions (H4) and (H6)\,(i) are Markov-type conditions similar to the ones considered by La\"{\i}b and Louani (\cite{LaibLouani:10}, \cite{LaibLouani:11}) (c.f. (A3) (i), (ii)),  Chaouch and Khardani (\cite{ChaKhar:15}) (c.f. (A4), p. 69) and Chaubey et al. (\cite{ChauLaiSen:12}) (c.f. (A4)\,(i),\,(ii), (A5)\,(ii), p. 977).  To derive the asymptotic distribution of $R_n$ we use the combination of regularity conditions  concerning the density, the regression function and higher conditional moments (namely (H6)\,(ii),\,(iii) and (H7)) taken from Shi and Song (\cite{ShiSong:15}) (c.f. (A2), (A3) and (A4), p. 3492) and Chaubey  et al. (\cite{ChauLaiSen:12}) (c.f. (A5)\,(ii) and (A6), p. 977).

(H8) is implied by the dependence condition considered by Chaubey et al. (\cite{ChauLaiSen:12}) (c.f. (A7), p. 977), which was introduced Wu (\cite{Wu:03}) (c.f. Lemma 3, p. 13) as an alternative to the usual mixing conditions. It is satisfied by several linear and nonlinear time series, as shown by the authors. The interested reader is also referred to Huang et al. (\cite{HuangChenWu:14}) for a more detailed discussion on the so-called predictive dependence measures related to this hypothesis.

The conditions required on the bandwidth in Theorems 3.2 and 3.3 cor\-respond to those of Shi and Song (c.f.(\cite{ShiSong:15}), theorems 3.2 and 3.4, p. 3493 and 3494, respectively). In spite of being more restrictive than the previous ones, our hypotheses are classical in dependence settings such as mixing. In order to assure the condition imposed on $\left(h_n\right)_{n \in \N}$ in Theorem \ref{convsupremo} we may take, for instance, $h_n= n^{-\alpha}$, with $\alpha<\theta$. As for Theorems \ref{TeorTLCdensity} and \ref{asympNormal}, a possible choice is $h_n= n^{-\alpha}$, with $\frac{2}{5}<\alpha<2$ in a) and $\frac{1}{3}<\alpha<\frac{1}{2}$ in b).

\section{Appendix} \label{sec4}

Firstly let us introduce some further notations and present two essential equalities that will be needed for the proofs.

For $x \geqslant 0$, we  write\ \ $R_n(x)= \frac{N_n(x)}{D_n(x)}$,\ \ where \vspace*{0.1cm}

\centerline{$N_n(x) = \frac{1}{\,n\,} \overset{n}{\underset{t=1}{\sum }}\,\Phi(Y_t)\, K_{_{\alpha(n,x),\beta(n)}}\!(X_t)$}\noindent and $D_n$ is defined by (\ref{defDn}). Futhermore, consider

\hspace*{0.25cm}  $\overline{N}_n(x) =  \frac{1}{\,n\,} \overset{n}{\underset{t=1}{\sum }}\,E\!\left(\Phi(Y_t)\, K_{_{\alpha(n,x),\beta(n)}}\!(X_t)/ \mathcal{F}_{t-1}\right)$\ \ and 

\hspace*{0.25cm}   $\overline{D}_n(x) =  \frac{1}{\,n\,} \overset{n}{\underset{t=1}{\sum }}\,E\!\left(K_{_{\alpha(n,x),\beta(n)}}\!(X_t)/ \mathcal{F}_{t-1}\right).$ \vspace*{0.2cm}

Observe that, under hypotheses (H4) and (H6)\,(i), a routine argument and the properties of conditional expectation lead to
\noindent \begin{itemize}
  \item[(i)] $E\!\left(\Phi(Y_t)\,K_{_{\alpha(n,x),\beta(n)}}\!(X_t)/\mathcal{F}_{t-1}\right)\! = \!E\!\left(E\!\left(\Phi(Y_t)/\mathcal{G}_{t-1}\right)\,K_{_{\alpha(n,x),\beta(n)}}\!(X_t)/\mathcal{F}_{t-1}\right) \vspace*{0.25cm} $\\
       \hspace*{5.25cm}$ \!= \!E\!\left(R(X_t)\,K_{_{\alpha(n,x),\beta(n)}}\!(X_t)/\mathcal{F}_{t-1}\right), \ t\in\N;$
  \item[(ii)] $E\!\left((\Phi(Y_t)- R(X_t))^{^2}K^{^2}_{_{\alpha(n,x),\beta(n)}}\!(X_t)/\mathcal{F}_{t-1}\right)=$ \vspace*{0.1cm}

\hfill $\begin{array}{rl}
& = E\!\left(E\!\left((\Phi(Y_t)- R(X_t))^{^2}/\mathcal{G}_{t-1}\right)K^{^2}_{_{\alpha(n,x),\beta(n)}}\!(X_t)/\mathcal{F}_{t-1}\right)
\end{array}$

\qquad $\begin{array}{rl}
   & = E\!\left((W_2(X_t)- R^{^2}(X_t))K^{^2}_{_{\alpha(n,x),\beta(n)}}\!(X_t)/\mathcal{F}_{t-1}\right) \vspace*{0.25cm}\\
   & = E\!\left(\sigma^{2}(X_t)K^{^2}_{_{\alpha(n,x),\beta(n)}}\!(X_t)/\mathcal{F}_{t-1}\right),\ \ t\in\N.
 \end{array}$
\end{itemize}

We can now present the proofs of the referred theorems.

Let us mention that all the constants appearing hereafter will be denoted generically by $C$.

\subsection{Proofs of main results} \label{DemTeoremas}

\subsubsection{Proof of Theorem \ref{convsupremo}}
\label{DemSupReg}
\noindent We have
$$\begin{array}{rl}
               \underset{x\in \Delta}{\sup}\left|R_n(x) - R(x)\right|\! & \leqslant \left({\underset {x\in \Delta}{\inf}}|D_n(x)|\right)^{-1}\left\{\underset{x\in \Delta}{\sup}\left|N_n(x) - R(x)f(x)\right|+ \right.\vspace*{0.25cm} \\
               & \left.\ \ \ +\ {\underset {x\in \Delta}{\sup}}|R(x)|\ \underset{x\in \Delta}{\sup}\left|D_n(x) - f(x)\right|\right\}.
             \end{array}$$
\noindent Since $${\underset {x\in \Delta}{\inf}}|D_n(x)| \geqslant {\underset {x\in \Delta}{\inf}}f(x)- \underset{x\in \Delta}{\sup}\left|D_n(x) - f(x)\right| $$and $I={\underset {x\in \Delta}{\inf}}f(x)>0,$ it suffices to prove, by Theorem \ref{convsupdens}, that \vspace*{-0.2cm}
\begin{equation}\label{sup1}
 \underset{x\in \Delta}{\sup}\left|N_n(x) - R(x)f(x)\right|\ \overset{a.s.}{\overrightarrow {\underset{n \rightarrow + \infty}}}\ 0.
\end{equation}
\noindent To this end, we remark that
$$\underset{x\in \Delta}{\sup}\left|N_n(x) - R(x)f(x)\right|  \leqslant A_n + B_n,\vspace*{-0.15cm}$$
\noindent with$$
     A_n = \underset{x\in \Delta}{\sup}\left|N_n(x) - \overline{N}_n(x)\right|\ \ \ \ \mbox{and}\ \ \ \ \
     B_n = \underset{x\in \Delta}{\sup}\left|\overline{N}_n(x) - R(x)f(x)\right|.$$
But \vspace*{-0.1cm}
$$\begin{array}{rl}
    B_n\! & \leqslant \underset{x\in \Delta}{\sup} \left|{\displaystyle{\int_{0}^{+ \infty} }} R(y)K_{_{\alpha(n,x),\beta(n)}}\!(y)\,\left(\frac{1}{\,n\,}\overset{n}{\underset{t=1}{\sum }}\,f^{\mathcal{F}_{t-1}}(y)-f(y)\right)\,dy\, \right|+ \vspace*{0.25cm} \\
          &\ \ \ +\ \underset{x\in \Delta}{\sup} \left|{\displaystyle {\int_{0}^{+ \infty}}} R(y)\,K_{_{\alpha(n,x),\beta(n)}}\!(y)\,f(y)\,dy - R(x)f(x)\, \right|  \end{array}$$
          and then
$$\begin{array}{rl}
    B_n\!& \leqslant \left\|\,\frac{1}{\,n\,}\overset{n}{\underset{t=1}{\sum }}\,f^{\mathcal{F}_{t-1}}-f\,\right\|_0\ \underset{y\in \R_0^+}{\sup} |R(y)|\ + \vspace*{0.25cm} \\
          &\ \ \ +\ \underset{x\in \Delta}{\sup} \left|{\displaystyle {\int_{0}^{+ \infty}}} R(y)\,K_{_{\alpha(n,x),\beta(n)}}\!(y)\,f(y)\,dy - R(x)f(x)\, \right|.
  \end{array}$$ By (H2), the first term of the last sum tends $a.s.$ to zero. On the other hand, the uniform continuity of $Rf$ on $\Delta$ and Lemma \ref{lema2Hille} assure the convergence to zero of the second term.

\noindent In what concerns $A_n$, it is bounded by $A_n^+ + A_n^-$, with $A_n^\pm= \underset{x\in \Delta}{\sup}\left|\overset{n}{\underset{t=1}{\sum }}Z_{t,n}^\pm(x)\right|$, where \begin{equation}\label{Ztn+-}
Z_{t,n}^\pm (x)= \frac{1}{\,n\,} \left\{\Phi^\pm(Y_t) K_{_{\alpha(n,x),\beta(n)}}\!(X_t)- E\!\left(\Phi^\pm(Y_t)\, K_{_{\alpha(n,x),\beta(n)}}\!(X_t)/\mathcal{F}_{t-1}\right)\right\} \end{equation}
and
\begin{equation}\label{phi+-}
 \Phi^+(Y_t) = \Phi(Y_t)\ \I_{\left\{|\Phi(Y_t)|\geqslant M_t\right\}},\ \ \ \Phi^-(Y_t) = \Phi(Y_t)-\Phi^+(Y_t),
\end{equation}with $M_t= t^k,\ \ k={\frac{\,1-\theta\,}{2}},\ t\in \N.$

\noindent Hence, by Lemma \ref{convAn+}, the $a.s.$ convergence of $\left(A_n\right)_{n \in \N}$ to zero reduces to showing that\ \ $A_n^-\ \overset{a.s.}{\overrightarrow{\underset{n\rightarrow + \infty}}}\ 0.$

\noindent With this purpose, let us consider $\delta_n= n^{-\lambda}$, $\lambda>\frac{\,3\,}{2}$, and
$$\nu_n = \left\{\begin{array}{ll}
                 \frac{\,b-a\,}{\delta_n}, &\ \ \frac{\,b-a\,}{\delta_n} \in \N \vspace*{0.2cm}\\
                 \left[\frac{\,b-a\,}{\delta_n}\right] + 1, &\ \ \frac{\,b-a\,}{\delta_n} \notin \N,
              \end{array}\right.
$$ where $ \left[u\right]$ denotes the integer part of the real number $u$.

\noindent Partitioning  $\Delta$ into the intervals $$\Delta_{j,n} = \left[a+(j-1) \delta_n,a+j \delta_n\right[,\ j \in \left\{1,\ldots,\nu_n - 1\right\},\ \ \Delta_{\nu_n,n}  = \left[a+(\nu_n - 1) \delta_n,b\right],$$ we may write
$$\textstyle A_n^-  =\underset{1\leqslant j \leqslant \nu_n}{\max}\ \underset{x \in \Delta_{j,n}}{\sup}\ \left|\frac{1}{\,n\,}\overset{n}{\underset{t=1}{\sum }}Z_{t,n}^-(x)\right|\ \leqslant\ A^-_{1,n} + A^-_{2,n},$$
\noindent with
$$\textstyle A^-_{1,n}\! =\!\! \underset{1\leqslant j \leqslant \nu_n}{\max}\ \underset{x \in \Delta_{j,n}}{\sup}\! \left|\frac{1}{\,n\,}\overset{n}{\underset{t=1}{\sum }}\!\left(Z_{t,n}^-\!(x)\!-\!Z_{t,n}^-\!\left(x_{j,n}\right)\right)\right|\,\ \mbox{and}\ \, A^-_{2,n}\! =\!\! \underset{1\leqslant j \leqslant \nu_n}{\max}\! \left|\frac{1}{\,n\,}\overset{n}{\underset{t=1}{\sum }}Z_{t,n}^-\!\left(x_{j,n}\right)\right|,$$

\vspace*{0.2cm}

\noindent $x_{j,n}$ being an arbitrary point in $\Delta_{j,n},\ j=1,\ldots,\nu_n$.

\vspace*{0.1cm}

\noindent As for $A^-_{1,n}$, note that, for sufficiently large $n$ and $j\in\{1,\ldots,\nu_n\}$,\vspace*{0.1cm}

\noindent $\left|\frac{1}{\,n\,}\ \overset{n}{\underset{t=1}{\sum }}\left(Z_{t,n}^-(x)-Z_{t,n}^-\left(x_{j,n}\right)\right)\right|\ \leqslant$

\noindent \hfill $\begin{array}{l}

                   \leqslant\ \frac{1}{\,n\,}\ \overset{n}{\underset{t=1}{\sum }}\left|\Phi^-(Y_t)\right|\,|K_{_{\alpha(n,x),\beta(n)}}\!(X_t)- K_{_{\alpha (n,x_{j,n}),\beta(n)}}\!(X_t)|\ + \vspace*{0.25cm}\\
                    \ \ \ +\ \frac{1}{\,n\,}\ \overset{n}{\underset{t=1}{\sum }}\,E\!\left(\left|\Phi^-(Y_t)\right|\,|K_{_{\alpha(n,x),\beta(n)}}\!(X_t)- K_{_{\alpha(n,x_{j,n}),\beta(n)}}\!(X_t)| /\mathcal{F}_{t-1}\right) \vspace*{0.25cm}\\
                   \leqslant\ \frac{\,C\,}{n}\ \overset{n}{\underset{t=1}{\sum }}\, M_t\ \frac{\ |x-x_{j,n}|\ }{\sqrt{h_n^3\,}},
\end{array}
$ \vspace*{0.25cm}

\noindent by the gamma kernel properties (c.f. Lemma \ref{Klipschiz}).

\noindent Applying the ergodic theorem, we have
\begin{center}
$\frac{\,1\,}{n}\ \overset{n}{\underset{t=1}{\sum }} \Big(E\!\left(\left|\Phi(Y_t)\right|\right)+E\!\left(\left|\Phi(Y_t)\right| /\mathcal{F}_{t-1}\right)\Big)\ \overset{a.s.}{\overrightarrow {\underset{n \rightarrow + \infty}}}\ 2\,E\left(\left|\Phi(Y_1)\right|\right).$
\end{center}
Thus, $\textstyle A^-_{1,n} = O\left(\frac{\,\delta_n\,}{\,\sqrt{h_n^3\,}\,}\right)\ a.s.$ and, taking into account the choice of $\delta_n$,
$$A^-_{1,n}\ \overset{a.s.}{\overrightarrow {\underset{n \rightarrow + \infty}}}\ 0.$$

\noindent Now we study the behaviour of $A^-_{2,n}$. In order to apply Azuma's inequality, we must find an upper bound for $\left|Z_{t,n}^-\left(x_{j,n}\right)\right|,\ t \in \{1,\ldots,n\},\ j \in \{1,\ldots,\nu_n\}$.

\noindent Using the fact that, for $x>0$,  
\begin{equation}\label{majK}
\textstyle K_{_{\alpha(n,x),\beta(n)}}\!(X_t) \leqslant \frac{C}{\,\sqrt{\,x\,h_n\,}\,} 
\end{equation}
(c.f. Shi and Song (\cite{ShiSong:15}), p. 3505, (5.20)), we get 
\begin{center}
$\forall\, t \in \{1,\ldots,n\},\ \ \forall\,x \in \Delta,\ \ \left|Z_{t,n}^- (x)\right| \leqslant\ 2\, |\Phi^-(Y_t)|\, K_{_{\alpha(n,x),\beta(n)}}\!(X_t)  \leqslant\ C\, \frac{M_n}{\,\sqrt{h_n\,}\,}.$
\end{center}
\noindent Consequently,
\noindent $$
\textstyle \forall\, \varepsilon > 0,\ \ P\!\left(\!A^-_{2,n} > \varepsilon\!\right)\leqslant \overset{\nu_n}{\underset{j=1}{\sum }} P\!\left(\Big|\overset{n}{\underset{t=1}{\sum }}Z_{t,n}^-\left(x_{j,n}\right)\Big| > n \varepsilon\!\right)=O\!\left(\nu_n\,\exp\left(\!-\,\frac{\,C\,n\,h_n\,}{\,M_n^2\,}\!\right)\right).
$$
\noindent The condition ${\underset{n \rightarrow + \infty}{\lim}} \frac{\,n^\theta\,h_n\,}{\,\log n\,} = +\infty$  yields the $a.s.$ convergence of $A^-_{2,n}$ to zero, as $n\rightarrow+\infty$, via the Borel-Cantelli lemma.

\subsubsection{Proof of Theorem \ref{convsupdens}}
\label{DemSupDen}

\noindent The proof is performed over the same steps that Theorem \ref{convsupremo} by taking $\Phi=1$ (and thus $R= 1$) and considering the same partition of $\Delta$. In this case, we obtain

\centerline{$\left|A^-_{1,n}\right| = O\left(\frac{\delta_n}{\,\sqrt{h_n^3\,}\,}\right)\ a.s.$\ \ and\ \ $P\left(A^-_{2,n} > \varepsilon\right) = O\left(\nu_n\, \exp\left(-C n h_n\right)\right)$,\ \ $\varepsilon>0$.}

\subsubsection{Proof of Theorem \ref{asympNormal}}
\label{DemTLCReg}
\noindent Let us decompose
$$R_n(x) - R(x) = \frac{\ N_{1,n}(x)+N_{2,n}(x)\ }{D_n(x)},\ \ \ \ x\geqslant0,$$ where
\begin{equation}\label{N1n(x)}
  \displaystyle N_{1,n}(x) = (N_n(x) - R(x) D_n(x))-(\,\overline{N}_n(x)- R(x) \overline{D}_n(x)), \vspace*{-0.2cm}
\end{equation}
\begin{equation}\label{N2n(x)}
 N_{2,n}(x) = \overline{N}_n(x)- R(x) \overline{D}_n(x).
\end{equation}

\noindent The following notations will be used hereafter. For $t \in \{1,\ldots,n\}$,
$$ U_{t,n}(x)=V_{t,n}(x) - E\!\left(V_{t,n}(x)/\mathcal{F}_{t-1}\right),\ \ \ x \geqslant 0,$$ and
$$ V_{t,n}(x) = \left\{\begin{array}{ll}
                         \frac{\,\sqrt[4]{h_n}\,}{\,\sqrt{n}\,}\,\left(\Phi(Y_t)- R(x)\right)K_{_{\alpha(n,x),\beta(n)}}\!(X_t), &\ \ \ x>0 \vspace*{0.35cm}\\
                         \frac{\,\sqrt{h_n}\,}{\,\sqrt{n}\,}\,\left(\Phi(Y_t)- R(0)\right)K_{_{\alpha(n,0),\beta(n)}}\!(X_t), &\ \ \ x=0 .
                       \end{array}\right.
$$

\noindent a) Consider $x>0$.

\noindent Since\ \ $D_n(x)\ \overset{P}{\overrightarrow{\underset{n\rightarrow + \infty}}}\ f(x)$\ \ and\ \  $\sqrt{n\,\sqrt{h_n\,}\,}\ N_{2,n}(x)\overset{P}{\overrightarrow{\underset{n\rightarrow + \infty}}}\ 0$, by Lemmas \ref{ConvPDn(x)} a) and \ref{ConvN2n(x)} a), respectively, we only need to show that
\begin{equation}\label{ConvN1n}
 \textstyle \sqrt{n\,\sqrt{h_n\,}\,}\ N_{1,n}(x)\ \overset{\mathcal{L}}{\overrightarrow{\underset{n\rightarrow + \infty}}}\ \mathcal{N}\left(0, \frac{\ \sigma^2(x)\,f(x)\ }{2\,\sqrt{\pi\,x\,}}\right).
\end{equation}
\noindent The fact that $\sqrt{n\,\sqrt{h_n\,}\,}N_{1,n}(x) = \overset{n}{\underset{t=1}{\sum }}\,U_{t,n}(x)$ and, for each $n \in \N$, $\left(U_{t,n}(x)\right)_{t \in \{1,\ldots,n\}}$ is a martingale difference with respect to the filtration $\left(\mathcal{F}_{t-1}\right)_{t\in\{1,\ldots,n\}}$ allow us to apply the central limit theorem for discrete time martingales. So, according to Hall and Heyde ((\cite{HallHeide:80}), p. 58), we must prove that
\begin{equation}\label{(1)}
\textstyle\overset{n}{\underset{t=1}{\sum }}\,E\!\left(U_{t,n}^{^2}(x) / \mathcal{F}_{t-1}\right)\ \overset{P}{\overrightarrow{\underset{n\rightarrow + \infty}}}\ \frac{\ \sigma^{2}(x)\,f(x)\ }{2\,\sqrt{\pi\,x\,}},\vspace*{-0.2cm}
\end{equation}
\begin{equation}\label{(2)}
\textstyle \forall\, \varepsilon > 0,\ \ \ n\,E\!\left(U_{t,n}^{^2}(x)\,\mathbb{I}_{\left\{\left|U_{t,n}(x)\right|> \varepsilon \right\}}\right) = o(1).\end{equation}
As (\ref{(2)}) is established in Lemma \ref{TLCcond2} a), the proof is reduced to checking (\ref{(1)}).
Based on the equality $$E\!\left(U_{t,n}^{^2}(x)/\mathcal{F}_{t-1}\right)=E\!\left(V_{t,n}^{^2}(x)/\mathcal{F}_{t-1}\right) - E^{^2}\!\left(V_{t,n}(x)/\mathcal{F}_{t-1}\right),\ \ \ \ t = 1, \ldots,n,$$ it suffices to show that
\begin{equation}\label{1º}
 \textstyle \overset{n}{\underset{t=1}{\sum }}\,E\!\left(V_{t,n}^{^2}(x) / \mathcal{F}_{t-1}\right)\ \overset{P}{\overrightarrow{\underset{n\rightarrow + \infty}}}\ \frac{\ \sigma^{2}(x)\,f(x)\ }{2\,\sqrt{\pi\,x\,}}, \vspace*{-0.2cm}
\end{equation}
 \begin{equation}\label{2º}
 \textstyle \overset{n}{\underset{t=1}{\sum }}\,E^{^2}\!\left(V_{t,n}(x) / \mathcal{F}_{t-1}\right)\ \overset{P}{\overrightarrow{\underset{n\rightarrow + \infty}}}\ 0.
 \end{equation}
\noindent Let us begin by pointing out that
$$\begin{array}{rl}
E\!\left(V_{t,n}^{^2}(x) / \mathcal{F}_{t-1}\right)\!& =\frac{\,\sqrt{h_n\,}\,}{n}\,E\!\left((\Phi(Y_t)-R(X_t))^{^2}
                                                         K^{^2}_{_{\alpha(n,x),\beta(n)}}\!(X_t)/\mathcal{F}_{t-1}\right) +\vspace*{0.25cm}\\
                                                         &\ \ \ +\,\frac{\,\sqrt{h_n\,}\,}{n}\, E\!\left((R(X_t)-R(x))^{^2}K^{^2}_{_{\alpha(n,x),\beta(n)}}\!(X_t)/
                                                         \mathcal{F}_{t-1}\right),
 \end{array}
$$ since  $E\!\left(\Phi(Y_t)- R(X_t)/\mathcal{G}_{t-1}\right) = 0$ (c.f. (i), p. 9).

\vspace*{0.2cm}

\noindent Hence, by (ii) (c.f. p. 9),\ \ $\overset{n}{\underset{t=1}{\sum }}E\!\left(V_{t,n}^{^2}(x) / \mathcal{F}_{t-1}\right)$ is equal to 
\begin{center}
$\frac{\,\sqrt{h_n}\,}{n} \overset{n}{\underset{t=1}{\sum }}E\!\left(K^{^2}_{_{\alpha(n,x),\beta(n)}}\!(X_t)\Big(\!\sigma^{2}(X_t) \!+ \!(R(X_t)\!-\!R(x))^{^2}\Big)/\mathcal{F}_{t-1}\right)
 \! =\! L_{1,n}(x) + L_{2,n}(x),$
\end{center} \vspace*{0.1cm}

\noindent with

\vspace*{0.1cm}

\noindent  $L_{1,n}(x) \!= \!\sqrt{h_n}\! {\displaystyle \int_{0}^{+\infty}}\!\!\!K^{^2}_{_{\alpha(n,x),\beta(n)}}\!(y)\!\left(\!\sigma^2(y)\!+\!(R(y)\!-\!R(x))^{^2}\!\right)\!
                                                         \left(\!\frac{1}{\,n\,}\overset{n}{\underset{t=1}{\sum }}\!f^{\mathcal{F}_{t-1}}(y)\!-\!f(y)\!\right)dy,$\vspace*{0.1cm}

\noindent $L_{2,n}(x) = \sqrt{h_n\,}\, {\displaystyle \int_{0}^{+\infty}}K^{^2}_{_{\alpha(n,x),\beta(n)}}\!(y)\left(\sigma^2(y)+(R(y)-R(x))^{^2}\right)f(y)\,dy.$

\vspace*{0.25cm}

\noindent Next, note that $\left|L_{1,n}(x) \right|$ is bounded by \vspace*{0.2cm}
\begin{center}
$
\left\|\! \frac{1}{\,n\,}\overset{n}{\underset{t=1}{\sum }}\,f^{\mathcal{F}_{t-1}}-f\right\|_0\!\! \sqrt{h_n\,}\,B(2,n,x)\, E\!\left(\sigma^{2}\!\left(G_{(2,n,x)}\right) \!+ \!\left(R\!\left(G_{(2,n,x)}\right)\!-\!R(x)\right)^{^2}\right)
$
\end{center} \vspace*{0.2cm}
and we may rewrite $L_{2,n}(x)$ in the form
$$\textstyle
 L_{2,n}(x)\! = \! \sqrt{h_n\,}\,B(2,n,x)\,E\!\left(\left(\sigma^{2}\!\left(G_{(2,n,x)}\right) + \left(R\!\left(G_{(2,n,x)}\right)-R(x)\right)^{^2}\right)f\!\left(G_{(2,n,x)}\right)\right).
$$

\noindent From (H6)(iii), Lemmas \ref{lema2Hille} and \ref{lema1} b), it follows that $\sqrt{h_n\,}\,B(2,n,x)\,{\overrightarrow{\underset{n\rightarrow + \infty}}}\,\frac{1}{\,2\,\sqrt{\pi\,x\,}\,}$ and $$\underset{n\rightarrow + \infty}{\lim}E\!\left(\left(\sigma^{2}\!\left(G_{(2,n,x)}\right) + \left(R\!\left(G_{(2,n,x)}\right)-R(x)\right)^{^2}\right)\,f\!\left(G_{(2,n,x)}\right)\right)=\sigma^{2}(x)\,f(x),$$
which concludes the proof of (\ref{1º}), by hypothesis (H2).

\noindent In order to obtain (\ref{2º}), we shall prove the convergence in mean to zero of the same sequence. As  \vspace*{-0.2cm}
\begin{center}
$\overset{n}{\underset{t=1}{\sum }}\,E^{^2}\!\left(V_{t,n}(x) / \mathcal{F}_{t-1}\right)\leqslant \overset{n}{\underset{t=1}{\sum }}\,\frac{\,\sqrt{h_n\,}}{n}\, E\!\left(\left[R(X_t)- R(x)\right]^{^2}K^{^2}_{_{\alpha(n,x),\beta(n)}}\!(X_t)/\mathcal{F}_{t-1}\right),$
\end{center}

\noindent we have \vspace*{0.2cm}

\noindent$\begin{array}{rl}
E\!\left(\overset{n}{\underset{t=1}{\sum }}E^{^2}\!\left(V_{t,n}(x) / \mathcal{F}_{t-1}\right)\!\right)\!\!& \leqslant \sqrt{h_n\,}\, {\displaystyle { \int_{0}^{+\infty}}} \,(R(y)- R(x))^{^2}\,K^{^2}_{_{\alpha(n,x),\beta(n)}}\!(y)\ f(y)\,dy
                                          \vspace*{0.25cm}\\
&=\sqrt{h_n\,} \, B(2,n,x)\,E\left( \!\left(R\!\left(G_{(2,n,x)}\right)\!-\! R(x)\right)\!^{^2}\,f\!\left(G_{(2,n,x)}\right)\right)\!.
\end{array}$

\vspace*{0.2cm}

\noindent Once again, the announced result is implied by Lemmas \ref{lema2Hille} and \ref{lema1} b).

\noindent {\bf b)} For $x=0$, the proof follows the same lines as the previous one by replacing $\sqrt{h_n\,}$ by $h_n$, noting that $B(2,n,0) = \frac{1}{\,2\,h_n\,}$ and applying Lemma \ref{lema2Hille}.

\subsubsection{Proof of Theorem \ref{TeorTLCdensity}}

\noindent Let us consider the decomposition
\begin{center}
$D_n(x) - D(x) = D_{1,n}(x)+D_{2,n}(x),\ \ \ \ x\geqslant0,$
\end{center} \vspace*{-0.25cm}
 where
\begin{equation}\label{(3)}
D_{1,n}(x) = D_n(x) - \overline{D}_n(x),\ \ \ \ \ D_{2,n}(x) = \overline{D}_n(x)-f(x).
\end{equation}

\noindent Again, our strategy is to prove the asymptotic normality of $\left(\sqrt{n\,\sqrt{h_n\,}}\,D_{1,n}(x)\right)_{n\in \N}$ and that \ $\sqrt{n\,\sqrt{h_n\,}}\,D_{2,n}(x)\ \overset{P}{\overrightarrow{\underset{n\rightarrow + \infty}}}\ 0$.

\noindent In view of Remark \ref{D2,n(x).Dens}, we proceed with the study of  the sequence involving $\,D_{1,n}(x)$.

\noindent As before, $ U_{t,n}(x)=V_{t,n}(x) - E\!\left(V_{t,n}(x)/\mathcal{F}_{t-1}\right)$,\ \ $x \geqslant 0,$  considering
\begin{equation}\label{Vtn(x)DENSIDADE}
 V_{t,n}(x) = \left\{\begin{array}{ll}
                         \frac{\,\sqrt[4]{h_n}\,}{\,\sqrt{n}\,}\,K_{_{\alpha(n,x),\beta(n)}}\!(X_t), &\ \ \ x>0 \vspace*{0.35cm}\\
                         \frac{\,\sqrt{h_n}\,}{\,\sqrt{n}\,}\,K_{_{\alpha(n,0),\beta(n)}}\!(X_t), &\ \ \ x=0 .
                       \end{array}\right.
\end{equation}

\noindent a) Let $x>0$.

\noindent To apply the central limit theorem for discrete time martingales, we write \linebreak $\sqrt{n\,\sqrt{h_n\,}\,}D_{1,n}(x) = \overset{n}{\underset{t=1}{\sum }}\,U_{t,n}(x)$ and, invoking Remark \ref{Vt,n.Dens}, we restrict ourselves to showing that
$$\textstyle \overset{n}{\underset{t=1}{\sum }}\,E\!\left(U_{t,n}^{^2}(x) / \mathcal{F}_{t-1}\right)\ \overset{P}{\overrightarrow{\underset{n\rightarrow + \infty}}}\ \frac{\,f(x)\,}{\ 2\,\sqrt{\pi\,x\,}\ }.$$

\noindent Once more, this result follows from
\begin{equation*} \textstyle
  \overset{n}{\underset{t=1}{\sum }}\,E\!\left(V_{t,n}^{^2}(x) / \mathcal{F}_{t-1}\right)\ \overset{P}{\overrightarrow{\underset{n\rightarrow + \infty}}}\ \frac{\,f(x)\,}{2\,\sqrt{\pi\,x\,}}\ \ \ \textup{and}\ \ \
  \overset{n}{\underset{t=1}{\sum }}\,E^{^2}\!\left(V_{t,n}(x) / \mathcal{F}_{t-1}\right)\ \overset{P}{\overrightarrow{\underset{n\rightarrow + \infty}}}\ 0.
 \end{equation*}
As in Theorem 3.3, we then write   $\,\overset{n}{\underset{t=1}{\sum }}E\!\left(V_{t,n}^{^2}(x) / \mathcal{F}_{t-1}\right) = D^a_{1,n}(x) + D^b_{1,n}(x),$ with

$D^a_{1,n}(x) = \sqrt{h_n}\, {\displaystyle \int_{0}^{+\infty}} K^{^2}_{_{\alpha(n,x),\beta(n)}}\!(y)
                                                         \left(\!\frac{1}{\,n\,}\overset{n}{\underset{t=1}{\sum }}\,f^{\mathcal{F}_{t-1}}(y)-f(y)\right)dy,$

$D^b_{1,n}(x) = \sqrt{h_n\,}\, {\displaystyle \int_{0}^{+\infty}}K^{^2}_{_{\alpha(n,x),\beta(n)}}\!(y)\,f(y)\,dy.$

\vspace*{0.2cm}

\noindent Noting that \vspace*{-0.25cm}
\begin{center}
$\left|D^a_{1,n}(x) \right| \leqslant \sqrt{h_n\,}\ \left\|\frac{1}{\,n\,}\overset{n}{\underset{t=1}{\sum }}\,f^{\mathcal{F}_{t-1}}-f\right\|_0 \ \underset{y\in \R_0^+}{\sup} K_{_{\alpha(n,x),\beta(n)}}\!(y) = o(1)\ \ \ a.s.,$
\end{center} \vspace*{-0.5cm}
by (\ref{majK}), and \vspace*{-0.2cm}
\begin{center}
$ D^b_{1,n}(x) = \sqrt{h_n\,}\,B(2,n,x)\,E\!\left(f\left(G_{(2,n,x)}\right)\right)\ {\overrightarrow{\underset{n\rightarrow + \infty}}}\ \frac{\,f(x)\,}{\ 2\,\sqrt{\pi\,x\,}\ },$\vspace*{-0.2cm}
\end{center}
we have $$\textstyle D^a_{1,n}(x) + D^b_{1,n}(x)\ {\overrightarrow{\underset{n\rightarrow + \infty}}}\ \frac{\,f(x)\,}{\ 2\,\sqrt{\pi\,x\,}\ }\ a.s.. $$
\noindent In order to study the behaviour of $\,\,\left( \overset{n}{\underset{t=1}{\sum }}\,E^{^2}\!\left(V_{t,n}(x) / \mathcal{F}_{t-1}\right)\right)_{n \in \N}\,,$ observe that the Cauchy-Schwarz inequality leads to \vspace*{-0.1cm}
\begin{center}
$\begin{array}{rl}
\overset{n}{\underset{t=1}{\sum }}\,E^{^2}\left(V_{t,n}(x) / \mathcal{F}_{t-1}\right)&\leqslant \frac{\sqrt{h_n}}{\,n\,}\,\overset{n}{\underset{t=1}{\sum }}\,\displaystyle {\int_{0}^{+\infty}}K_{_{\alpha(n,x),\beta(n)}}(y)\,\left(f^{\mathcal{F}_{t-1}}(y)\right)^2\,dy \vspace*{0.25cm}\\
&= \sqrt{h_n} {\displaystyle \int_{0}^{+\infty}} K_{_{\alpha(n,x),\beta(n)}}(y)\,\left(\frac{1}{\,n\,}\overset{n}{\underset{t=1}{\sum }} \left(f^{\mathcal{F}_{t-1}}(y)\right)^2\right)\,dy.
\end{array}$
\end{center}
Therefore, denoting by $\,f^* \in C_0(\R)\,$ the limit of $\,\left(\frac{1}{\,n\,}\overset{n}{\underset{t=1}{\sum }}\,\left(f^{\mathcal{F}_{t-1}}\right)^2\right)_{\!\!n \in \N}\,,$ the last term is bounded above by
$$\textstyle \sqrt{h_n}\left\| \frac{1}{\,n\,}\overset{n}{\underset{t=1}{\sum }}\,\left(f^{\mathcal{F}_{t-1}}\right)^2-f^*\right\|_0 +\sqrt{h_n\,}\left|E\left(f^*\left(G_{(1,n,x)}\right)\right)\right|,$$
\noindent which tends to zero, as $n \rightarrow +\infty,\,$ by Lemmas 4.2 and 4.1 b).

\noindent b) For $x=0$, the proof is straightforward by making the adequate substitutions.

 \subsection{Auxiliary results}
Consider $p$, $n$ and $x$ arbitrary numbers in $\left]0,+\infty\right[$, $\N$ and $\R_0^+$, respectively. For convenience, $G_p \equiv G_{(p,n,x)}$ will denote, in the sequel, a real random variable following the gamma distribution with shape parameter $ \frac{\,p\,x\,}{h_n}+1$ and scale parameter $\frac{\,h_n\,}{p}$.

\begin{lemma} \label{lema1}
Let $\varphi$ be a real function defined on $\R_0^+$ and $T$ a real random variable with density function $g$. If $p \in [1,+\infty[$ and $E\!\left( \varphi(T) K_{_{\alpha(n,x),\beta(n)}}^{^p} \!(T) \right)$ exists, we have: \vspace*{-0.1cm}
\begin{enumerate}
 \item[a)] $E\!\left( \varphi(T) K_{_{\alpha(n,x),\beta(n)}}^{^p} \!(T) \right) = B(p,n,x)\,E\!\left( \varphi(G_p) g(G_p) \right)$,\ \ where \vspace*{-0.3cm} \begin{center}
$B(p,n,x) = \frac{\Gamma\!\left(\frac{\,p\,x\,}{h_n}+1\right)}{\Gamma^{^{p}}\!\left(\frac{x}{\,h_n\,}+1\right) p^{\frac{\,p\,x\,}{h_n}+1} \, h_n^{p-1}} ;$
\end{center}
 \item[b)] for $x>0$,\ \ \ $ \underset{n\rightarrow +\infty}{\lim}\,h_n^{\frac{\,p-1\,}{2}} B(p,n,x) = \frac{1}{\,\sqrt{p\,}\,\left(\sqrt{2\,\pi\,x\, \,}\right)^{p-1}\,} $.
\end{enumerate}
\end{lemma}
\begin{proof} The proof of a) is trivial. As for b), the proof is performed over the same steps of Chen (c.f. (\cite{Chen:00}), p. 474, (3.2) and (3.3)), noting that $$\textstyle \forall\, p \in \left[1,+\infty \right[,\ \ \forall x>0,\ \ \ B(p,n,x) = \frac{S^{^{\,p}}\!\!\left(\frac{\,x\,}{h_n}\right)}
                  {\,S\!\left(\frac{\,p\,x\,}{h_n}\right)\, \left(\sqrt{2\,\pi\,x\,h_n\,}\right)^{p-1}\,\sqrt{p\,}\,} ,$$
with\ $S(z) = \frac{\,\sqrt{2\,\pi\,}\,e^{-z}\,z^{z+\frac{1}{\,2\,}}\,}{\Gamma(z+1)},\ z \!\geqslant\! 0$,\ and taking into account the properties of $S$. \vspace*{0.25cm}
\end{proof}

\begin{remark}\textup{ Notice that, for $x=0$, the result corresponding to b) is obvious since\ \ $h_n^{p-1} B(p,n,0)= \frac{1}{\,p\,}$.} \vspace*{0.25cm}
\end{remark}

\begin{lemma}\label{lema2Hille} If $p>0$, we have $\underset{n\rightarrow +\infty}{\lim}E(\ell(G_{(p,n,x)})) = \ell(x)$, for every real function $\ell$ defined on $\R_0^+$, continuous and bounded.
The convergence is uniform in every interval in which $\ell$ is uniformly continuous.\vspace*{-0.25cm}
\end{lemma}
\begin{proof} Please see Chaubey et al ((\cite{ChauLaiSen:12}), p. 975) and Feller ((\cite{Feller:70}), p. 219).\vspace*{0.25cm}
\end{proof}

\begin{lemma}\label{lema3} (Shi and Song (\cite{ShiSong:15}), p. 3491, 3501)\ \ Under assumption (H7) we have, for all $x \geqslant 0$,  \vspace*{-0.1cm}
 \begin{enumerate}
   \item[a)] $E\!\left(K_{_{\alpha(n,x),\beta(n)}}\!(X_1)\right) = f(x) + \frac{\,2\,f'(x) + x\, f''(x)\,}{2}\,h_n + o\left(h_n\right)$;
   \item[b)] $E\!\left(R(X_1)\,K_{_{\alpha(n,x),\beta(n)}}\!(X_1)\right)-R(x) E(K_{_{\alpha(n,x),\beta(n)}}\!(X_1) ) = $ \vspace*{0.25cm}

    \hfill $=\left[R'(x)\,f(x) + \frac{\,x\,}{2}\,R''(x)\,f(x) + x\,R'(x)\, f'(x)\,\right]\,h_n + o\left(h_n\right).$\endproof
 \end{enumerate} \vspace*{0.25cm}
\end{lemma}

\begin{lemma}\label{Klipschiz} The gamma kernel has the following property
\begin{center}
$\exists\,C>0:\ \ \forall\,y\in\R^+,\ \ \forall\,x,u \in \Delta,\
  \left|K_{_{\alpha(n,x),\beta(n)}}\!(y)\! -\! K_{_{\alpha(n,u),\beta(n)}}\!(y)\right|\! \leqslant\! \frac{\,C\,\left|x-u\right|\,}{\sqrt{h_n^3\,}},$\vspace*{-0.3cm}
\end{center} for sufficiently large $n$.
\end{lemma}
\begin{proof}
  Consider\ \ $n \in \N$,\ \ $y>0$\ \ and\ \ $x,u \in \Delta$\ \ arbitrarily fixed, with $x>u$.

 \noindent Using a similar argument as Shi and Song ((\cite{ShiSong:15}), p. 2506), we make a Taylor expansion of $K_{_{\alpha(n,x),\beta(n)}}\!(y)$ at $x=u$ up to the first order:

\vspace*{0.2cm}

 \noindent $K_{_{\alpha(n,x),\beta(n)}}\!(y)-K_{_{\alpha(n,u),\beta(n)}}\!(y) = \frac{x-u}{\,h_n^2\,\Gamma\left(\frac{\tilde{x}}{h_n}+1\right)\,}\left(\frac{y}{h_n}\right)^{\frac{\tilde{x}}{h_n}}\! e^{-\frac{y}{\,h_n\,}}\!\left[ \log\!\left(\!\frac{y}{\,h_n\,}\!\right)\!-\!\Psi\!\left(\frac{\tilde{x}}{h_n}+1\right) \right]\!,$

\vspace*{0.2cm}

\noindent where $\Psi$ is the digamma function and $\tilde{x} \in ]u,x[$.

\noindent Now, by Stirling's formula, $\Psi$ properties and some algebraic manipulations, the second member of the previous equality takes the form:
\noindent \begin{center}
 $\frac{x-u}{\sqrt{2\,\pi\, \tilde{x}\,h_n^3\,}}\exp\left\{\frac{\tilde{x}}{h_n} \!\left( 1-\frac{y}{\,\tilde{x}\,}+ \log \left(\frac{y}{\,\tilde{x}\,} \right) \right)\right\}\!\left[ \log\!\left(\!\frac{y}{\,\tilde{x}\,}\!\right)\!+\! \log\!\left(\!\frac{\tilde{x}}{\,h_n\,}\!\right)\!-\!\Psi\!\left(\!\frac{\tilde{x}}{\,h_n\,}\!\right)\!-\!\frac{\,h_n\,}{\tilde{x}} \right]\,(1\!+o(1)).$\vspace*{-0.3cm}
\end{center}

\vspace*{0.2cm}

\noindent Therefore, $\left|K_{_{\alpha(n,x),\beta(n)}}\!(y)-K_{_{\alpha(n,u),\beta(n)}}\!(y)\right|$ is bounded by

\noindent \begin{center}
 $\frac{|x-u|}{\sqrt{2\,\pi\, \tilde{x}\,h_n^3\,}}\,\exp\left\{\frac{\tilde{x}}{h_n} \!\left( 1-\frac{y}{\,\tilde{x}\,}+ \log \left(\frac{y}{\,\tilde{x}\,} \right) \right)\right\}\!\Big[\, \left|\log\!\left(\!\frac{y}{\,\tilde{x}\,}\!\right)\right|\!+\!\,\frac{\,2 \,h_n\,}{\tilde{x}} \Big]\,(1\!+o(1)).$
 \end{center}

\noindent Taking into account that \ $1-\frac{y}{\,\tilde{x}\,}+ \log \left(\frac{y}{\,\tilde{x}\,} \right) \leqslant 0$\ \ and\ \ $\tilde{x}>a$, we can see that
$\frac{\,2 \,h_n\,}{\tilde{x}}\exp\left\{\frac{\tilde{x}}{h_n} \!\left( 1-\frac{y}{\,\tilde{x}\,}+ \log \left(\frac{y}{\,\tilde{x}\,} \right) \right)\right\}< \frac{\,2 \,h_n\,}{a}$.

\vspace*{0.2cm}

\noindent On the other hand, for sufficiently large $n$,
\noindent \begin{center}
  $\left|\log\!\left(\!\frac{y}{\,\tilde{x}\,}\!\right)\right|\exp\left\{\frac{\tilde{x}}{h_n} \!\left( 1-\frac{y}{\,\tilde{x}\,}+ \log \left(\frac{y}{\,\tilde{x}\,} \right) \right)\right\}\! < e^{s_n(y)},
  $\end{center} \vspace*{-0.25cm}
 where \vspace*{-0.25cm}
\noindent \begin{center}
   $s_n (y)= \left\{\begin{array}{ll}
                   \left(\frac{b}{\,h_n\,} + 1\right) \log\left( 1+ \frac{\,h_n\,}{a}\right) -1,&\ \ \ y \geqslant 1 \vspace*{0.25cm}\\
                   \left(\frac{a}{\,h_n\,} - 1\right) \log\left( 1- \frac{\,h_n\,}{b}\right) +1, &\ \ \ 0 < y < 1.
                 \end{array}
   \right.$
 \end{center}
The convergence of the sequence $\left(h_n\right)_{n \in \N}$ to zero leads to the desired result.
\end{proof}

For the next lemmas recall the notations introduced in subsection \ref{DemTeoremas}, namely, (\ref{Ztn+-}), (\ref{phi+-}), (\ref{N2n(x)}) and (\ref{(3)}). \vspace*{0.25cm}

\begin{lemma}\label{convAn+}
If\ \ $E\!\left( \left|\Phi(Y_1)\right|^{\tau+1}\right) < + \infty$, for some $\tau > 0$, and the sequen\-ce  $(h_n)_{_{n\in\N}}$ satisfies \vspace*{-0.5cm}
\begin{center}
$n\,\sqrt{h_n\,}\ {\big{\uparrow}}+\infty,\ \ \ \ \ \exists\, \theta \in \left]0,\frac{\tau}{\,\tau+1\,}\right[\,:\ \ \ \displaystyle \underset{n \rightarrow + \infty}{\lim}\ \frac{\,n^{\theta}\,h_n\,}{\log n}\ =\ +\infty,$
\end{center}
then $$\textstyle \underset{x\in \Delta}{\sup}\left|\overset{n}{\underset{t=1}{\sum }}Z_{t,n}^+(x)\right|\ \overset{a.s.}{\overrightarrow {\underset{n \rightarrow + \infty}}}\ 0.$$
\end{lemma}
\begin{proof}Noting that, for $x>0$, $K_{_{\alpha(n,x),\beta(n)}}\!(X_t)$ is bounded by $\frac{C}{\,\sqrt{x\,h_n\,}\,}$\ \ (c.f. (\ref{majK})), we may write
$$\textstyle \underset{x\in \Delta}{\sup}\left|\frac{1}{\,n\,} \overset{n}{\underset{t=1}{\sum }}Z_{t,n}^+(x)\,\right|\ \leqslant \frac{C}{\,n\,\sqrt{a\,h_n\,}\,}\  \overset{n}{\underset{t=1}{\sum }}\left(\left|\Phi^+(Y_t) \right| + E\left(\left|\Phi^+(Y_t) \right|/ \mathcal{F}_{t-1}\right) \right).$$
As \ \ $n\,\sqrt{h_n\,}\ {\big{\uparrow}}+\infty$,\, it suffices to prove, by Kronecker's lemma, that
\begin{equation}\label{VanRyzin}
\textstyle \overset{+\infty}{\underset{t=1}{\sum }}\,\frac{1}{\,t\,\sqrt{\,h_t\,}\,}\left(\left|\Phi^+(Y_t)\right| + E\left(\left|\Phi^+(Y_t)\right|/ \mathcal{F}_{t-1}\right) \right)< +\infty,\ \ a.s..
\end{equation}

To this aim, define $$\textstyle W_n = \overset{n}{\underset{t=1}{\sum }}\,\frac{1}{\,t\,\sqrt{\,h_t\,}\,}\left(\left|\Phi^+(Y_t)\right| + E\left(\left|\Phi^+(Y_t)\right|/ \mathcal{F}_{t-1}\right) \right),\ \ n \in \N.$$
The sequence $\left(W_n\right)_{n \in \N}$ satisfies the conditions of Van Ryzin's lemma (c.f. Van Ryzin (\cite{VanRyzin:69}), p. 1765), in particular,
$\displaystyle E\left(W_{n+1}/\mathcal{F}_{n}\right)\ =\ W_n + W'_n,$\ \ with
$$\textstyle W'_n = \frac{2}{\,(n+1)\,\sqrt{\,h_{n+1}\,}\,}\ E\!\left(\left|\Phi^+(Y_{n+1})\right|/ \mathcal{F}_{n}\right).$$

\vspace*{-0.3cm}

Consequently, to prove (\ref{VanRyzin}), it is enough to show that\ \ $\overset{+\infty}{\underset{n=1}{\sum }}E\left(W'_n\right) < +\infty$.

\noindent Applying H\"{o}lder and Markov inequalities, with $p= \tau+1$ and $q = \frac{\,\tau + 1\,}{\tau}$, we get
$$\begin{array}{rl}
E\!\left(\left|\Phi^+(Y_{n})\right|\right)\! &  \leqslant \left(E\!\left(\left|\Phi(Y_{n})\,\right|^p\right)\right)^{\frac{1}{\,p\,}}\left(E\left(|\Phi(Y_{n})|^p\right)\right)^{\frac{1}{\,q\,}}  M_n^{-\frac{p}{\,q\,}}\vspace*{0.3cm} \\
    & =  E\left(|\Phi(Y_{n})|^{\tau+1}\right)\,n^{-\tau k}
  \end{array}$$
which implies that $$\textstyle \overset{+\infty}{\underset{n=1}{\sum }} E\left(W'_n\right)\leqslant C\ \overset{+\infty}{\underset{n=1}{\sum }} \frac{1}{\,n^{1+\tau k}\,\sqrt{\,h_{n}\,}\,} .$$ The fact that\ \ $\underset{n \rightarrow + \infty}{\lim}\ \frac{\ n^{\theta}\,h_n\ }{\log n}\ =\ +\infty$ and the definition of $k$ assure the announced result. \vspace*{0.25cm}
\end{proof}

\begin{lemma}\label{ConvPDn(x)} Assume that (H1) and (H2) are fulfilled. If \begin{itemize}
                                  \item[a)] $\underset{n\rightarrow+\infty}{\lim} n\,\sqrt{ h_n\,}= + \infty ,$ then\ \ $D_n(x)\ \overset{P}{\overrightarrow{\underset{n\rightarrow + \infty}}}\ f(x),\ \ \ x>0;$
                                  \item[b)]  $\underset{n\rightarrow+\infty}{\lim} n\, h_n= + \infty ,$ then\ \ $D_n(0)\ \overset{P}{\overrightarrow{\underset{n\rightarrow + \infty}}}\ f(0).$
                                \end{itemize}
\end{lemma}

\begin{proof} Assumptions (H1) and (H2) assure that $\textstyle\left(\overline{D}_{n}(x)\right)_{n\in \N}$ converges in pro\-ba\-bility to $f(x)$, for all $x \geq 0$. The result follows then if we prove that
$$\underset{n\rightarrow+\infty}{\lim} E\left((D_n(x)- \overline{D}_n(x))^2\right)=0,\,\,x \geq 0.$$
To this end, we write
$$\begin{array}{rl}
E\!\left((D_n(x)- \overline{D}_n(x))^2\right) \!\!\!\!& = \! \frac{1}{\,n^2\,}\!\overset{n}{\underset{t=1}{\sum }}E\left(\!\left(\!\,K_{_{\alpha(n,x),\beta(n)}}(X_t)\!-\!E\!\left(\,K_{_{\alpha(n,x),\beta(n)}}(X_t)/ \mathcal{F}_{t-1}\!\right)\!\right)^2\!\right)
\vspace*{0.25cm}\\
              & = \frac{1}{\,n^2\,}\overset{n}{\underset{t=1}{\sum }}\!\left(\!E\left(\!K^2_{_{\alpha(n,x),\beta(n)}}(X_t)-E^2\!\left(\!K_{_{\alpha(n,x),\beta(n)}}(X_t)/ \mathcal{F}_{t-1}\right)\right)\right) \vspace*{0.25cm}\\
              &  \leqslant  \frac{2}{\,n^2\,}\overset{n}{\underset{t=1}{\sum }}E\left(\!K^2_{_{\alpha(n,x),\beta(n)}}(X_t)\right)\! =  \frac{2}{\,n\,}{\displaystyle \int_{0}^{+\infty}}\!\! K^2_{_{\alpha(n,x),\beta(n)}}(y)f(y)dy \vspace*{0.35cm}  \\
              & = \frac{2}{\,n\,}\,B(2,n,x)\, E\!\left(f\!\left(G_{(2,n,x)}\right)\right).
                \end{array}$$
Hence, using Lemmas \ref{lema1} and \ref{lema2Hille}, we get

\noindent a) if $x>0$,\ \ $E\left((D_n(x)- \overline{D}_n(x))^2\right) \leqslant \frac{2}{\,n\sqrt{h_n}}\left(\frac{1}{\,\sqrt{\pi x\,}}+o(1)\right)\left(f(x)+o(1)\right),$\vspace*{0.2cm}

\noindent b) if $x=0$,\ \ $E\left((D_n(0)- \overline{D}_n(0))^2\right) \leqslant \frac{1}{\,n h_n\,}\left(f(0)+o(1)\right),$

\noindent achieving the proof of the lemma.

\end{proof}

\begin{lemma}\label{ConvN2n(x)} Suppose (H3) and (H8) are verified. If \begin{itemize}
                                  \item[a)] $\underset{n\rightarrow+\infty}{\lim} n\,\sqrt{ h_n\,}= + \infty$\ \ and\ \ $\underset{n\rightarrow+\infty}{\lim} n\,\sqrt{h_n^5\,}=0 ,$ then \vspace*{-0.2cm} $$\textstyle\sqrt{n\,\sqrt{h_n\,}\,}\ N_{2,n}(x)\ \overset{P}{\overrightarrow{\underset{n\rightarrow + \infty}}}\ 0,\ \ \ x>0;$$
                                  \item[b)]  $\underset{n\rightarrow+\infty}{\lim} n\, h_n= + \infty$\ \ and\ \ $\underset{n\rightarrow+\infty}{\lim} n\,h_n^3=0 ,$ then\ \ $\sqrt{n\,h_n\,}\ N_{2,n}(0)\ \overset{P}{\overrightarrow{\underset{n\rightarrow + \infty}}}\ 0.$
                                \end{itemize}
\end{lemma}
\begin{proof}
Let us begin by decomposing $N_{2,n} (x) $: $$N_{2,n} (x) = \overline{N}_n(x)- R(x) \overline{D}_n(x) = N^a_{2,n}(x) + N^b_{2,n}(x),$$ with
\begin{equation}\label{Na2n(x)}
    N^a_{2,n}(x) =  \overline{N}_n(x)- R(x) \overline{D}_n(x) - \left(E(N_n(x)) - R(x) E(D_n(x))\right), \vspace*{-0.2cm}
  \end{equation}
\begin{equation}\label{Nb2n(x)}
  N^b_{2,n}(x) =  E(N_n(x)) - R(x) E(D_n(x)).
\end{equation}

\noindent For all $x\geqslant0$, $$\begin{array}{rl}
        N^a_{2,n}(x)\! & = \frac{1}{\,n\,} \bigg[\,\overset{n}{\underset{t=1}{\sum }}\,E\left((R(X_t) - R(x))\,K_{_{\alpha(n,x),\beta(n)}}\!(X_t)/\mathcal{F}_{t-1}\right) -\bigg. \vspace*{0.25cm}\\
         & \bigg.\ \ \ \ -\, n\, E\left((R(X_t) - R(x))\,K_{_{\alpha(n,x),\beta(n)}}\!(X_t)\right)\bigg] \vspace*{0.25cm}\\
         & =\frac{1}{\,n\,}{\displaystyle \int_{0}^{+\infty}}\!\! K_{_{\alpha(n,x),\beta(n)}}\!(y)\,(R(y)-R(x)) \left(\overset{n}{\underset{t=1}{\sum }}f^{\mathcal{F}_{t-1}}(y) - n f(y)\right)dy.
      \end{array}$$
      Setting  $H_n(y) = \overset{n}{\underset{t=1}{\sum }}f^{\mathcal{F}_{t-1}}(y) - n f(y) $, we are led to
      $$\textstyle \left|N^a_{2,n}(x)\right| \leqslant 2\, {\underset{u \in \R_0^+}\sup} |R(u)|\ \frac{1}{\,n\,}{\displaystyle \int_{0}^{+\infty}}\!\! K_{_{\alpha(n,x),\beta(n)}}\!(y)\,\left|H_n(y)\right|\,dy. $$
      Therefore, Cauchy-Schwarz inequality and Fubini-Tonelli theorem  yield $$\begin{array}{rl}
             E\left(\left|N^a_{2,n}(x)\right|^2\right)\! & \leqslant \frac{C}{\,n^2\,} {\displaystyle
                                                           \int_{0}^{+\infty}}\!\!K_{_{\alpha(n,x),\beta(n)}}\!(y)\,
                                                           E\left(H^2_n(y)\right)\,dy \vspace*{0.25cm}\\
              &  \leqslant \frac{C}{\,n^2\,}\ {\underset{y \in \R_0^+}\sup}\left\|H_n(y)\right\|^2_2 =  O\left( \frac{1}{\,n\,}\right)
           \end{array}
$$
since $\left\|H_n(y)\right\|^2_2 = O(n)$, by hypothesis (H8).

\noindent a) Consider $x>0$. Applying Markov's inequality, we obtain $$\textstyle \forall\, \varepsilon>0,\ \ P\left(\sqrt{n\,\sqrt{h_n\,}\,}\left|N^a_{2,n}(x)\right|>\varepsilon\right)\  = O\left(\sqrt{h_n\,}\right),$$
          and thus $\sqrt{n\,\sqrt{h_n\,}\,}\ N^a_{2,n}(x)\ \overset{P}{\overrightarrow{\underset{n\rightarrow + \infty}}}\ 0$.

\noindent On the other hand, from Lemma \ref{lema3} b), we have $$\textstyle\sqrt{n\,\sqrt{h_n\,}\,}\ N^b_{2,n}(x) = \sqrt{n\,\sqrt{h_n^5\,}\,}\,b(x) + \sqrt{n\,\sqrt{h_n\,}\,}\, o(h_n),$$
\noindent with $$b(x)= R'(x)\,f(x)+\frac{\,x\,}{2}\,R''(x)\,f(x) + x\, R'(x)\,f'(x),$$ which concludes the proof of a).

\noindent b) For $x=0$, a similar reasoning conduces to $$\textstyle\forall\, \varepsilon>0,\ \ P\left(\sqrt{n\,h_n\,}\left|N^a_{2,n}(0)\right|>\varepsilon\right)\  = O\left(h_n\right)$$ and
          $$\textstyle\sqrt{n\,h_n\,}\ N^b_{2,n}(0) = \sqrt{n\,h_n^3\,}\,R'(0)\,f(0) + \sqrt{n\,h_n\,}\, o(h_n),$$ completing the proof of the lemma.
\end{proof}\vspace*{0.25cm}

\begin{remark} \label{D2,n(x).Dens} \textup{In a similar way, the convergence of $\textstyle\left(\!\sqrt{n\,\sqrt{h_n\,}}D_{2,n}(x)\!\right)_{n\in \N}$ in probability to zero, needed in the proof of Theorem 3.4 (c.f. (\ref{(3)})), is based on the decomposition
$$D_{2,n} (x) = D^a_{2,n}(x) + D^b_{2,n}(x),$$
 with 
$$D^a_{2,n}(x) =  \overline{D}_n(x)-E(D_n(x)) \qquad \mbox{and} \qquad D^b_{2,n}(x) =  E(D_n(x)) - f(x).$$
In this case, we simply get
$$\textstyle D^a_{2,n}(x) = \frac{1}{\,n\,}{\displaystyle\int_{0}^{+\infty}}\!\! K_{_{\alpha(n,x),\beta(n)}}\!(y)\,\left(\overset{n}{\underset{t=1}{\sum }}H_n(y) \right)dy$$
which leads again to the equality  $\,\,E\left(\left|D^a_{2,n}(x)\right|^2\right)=  O\left( \frac{1}{\,n\,}\right),\,$ under (H8).}

\textup{As for the bias term, it suffices to apply Lemma \ref{lema3} a), giving
$$\textstyle\sqrt{n\,\sqrt{h_n\,}\,}\ D^b_{2,n}(x) = \sqrt{n\,\sqrt{h_n^5\,}\,}\,\left(f'(x)+ \frac{xf''(x)}{2}\right)+\sqrt{n\,\sqrt{h_n\,}\,}\, o(h_n).$$}
\end{remark}\vspace*{0.25cm}

\begin{lemma} \label{TLCcond2}Suppose that conditions (H5) and (H6)(ii) hold. If $\left(h_n\right)_{n \in \N}$ is such that\vspace*{-0.1cm}
\begin{itemize}
\item[a)] $\underset{n\rightarrow+\infty}{\lim} n\,\sqrt{ h_n\,}= + \infty$, then \vspace*{-0.3cm}
\begin{equation*}
\forall\, \varepsilon > 0,\ \ \ n\,E\!\left(U_{t,n}^{^2}(x)\,\mathbb{I}_{\left\{\left|U_{t,n}(x)\right|> \varepsilon \right\}}\right) = o(1),\ \ \ x>0.
\end{equation*}

\item[b)]  $\underset{n\rightarrow+\infty}{\lim} n\, h_n= + \infty$, then \vspace*{-0.3cm}\begin{equation*}
\forall\, \varepsilon > 0,\ \ \ n\,E\!\left(U_{t,n}^{^2}(0)\,\mathbb{I}_{\left\{\left|U_{t,n}(0)\right|> \varepsilon \right\}}\right) = o(1).
\end{equation*}
 \end{itemize}
\end{lemma}

\begin{proof} Let $x\geqslant0$ and $\varepsilon>0$. By corollary 9.5.2 of Chow and Teicher ((\cite{ChowTeicher:88}), p. 131), we have
$$\textstyle n\,E\!\left(U_{t,n}^{^2}(x)\,\mathbb{I}_{\left\{\left|U_{t,n}(x)\right|> \varepsilon \right\}}\right) \leqslant 4\,n\,E\!\left(V_{t,n}^{^2}(x)\,\mathbb{I}_{\left\{\left|V_{t,n}(x)\right|> \frac{\,\varepsilon\,}{4} \right\}}\right).$$
Applying once again the inequalities of H\"{o}lder and Markov, with $p = \frac{\,\zeta\,}{2} + 1$  and $q = \frac{2}{\,\zeta\,}+ 1$, the right-hand side of the last inequality is bounded by 
$$\textstyle 4\, n\left\{E\!\left(\left|V_{t,n}(x)\right|^{^{2p}}\right)\right\}^{\frac{1}{\,p\,}}\left\{P\!\left(\left|V_{t,n}(x)\right|> \frac{\,\varepsilon\,}{4} \right)\right\}^{\frac{1}{\,q\,}}\leqslant 4 \, n\,E\!\left(\left|V_{t,n}(x)\right|^{^{2p}}\right)\,\left(\frac{\,\varepsilon\,}{4}\right)^{^{-{\frac{\,2p\,}{\,q\,}}}}.
$$

\noindent a) Consider $x>0$. Jensen's inequality allows us to write
$$\begin{array}{rl}
n \,E\!\left(\left|V_{t,n}(x)\right|^{^{2p}}\right)\!&\leqslant \frac{\,h_n^{^{\frac{\,p\,}{2}}}\,}{n^{^{p^{^{}}}}}\, 2^{^{p-1}}\, E\!\left(\left(|\Phi(Y_t)|^{^{2p}}+ |R(x)|^{^{2p}}\right)\,K^{^{2p}}_{_{\alpha(n,x),\beta(n)}}\!(X_t)\right) \vspace*{0.3cm}\\

               & =\ \frac{\,h_n^{^{\frac{\,p\,}{2}}}\,}{n^{^{p^{^{}}}}}\, 2^{^{p-1}}\,E\!\left(\left(E(|\Phi(Y_t)|^{^{2p}}/X_t)+ |R(x)|^{^{2p}}\right)\,K^{^{2p}}_{_{\alpha(n,x),\beta(n)}}\!(X_t)\right)\vspace*{0.3cm}\\
               & \leqslant \frac{\,h_n^{^{\frac{\,p\,}{2}}}\,}{n^{^{p^{^{}}}}}\, 2^{^{p-1}}\!\bigg(\,\underset{y\in\R_0^+}{\sup}W_{2p}(y) + |R(x)|^{^{2p}}\!\bigg) B(2p,n,x)\, E\!\left(f\!\left(G_{(2p,n,x)}\right)\right)\!.
           \end{array}$$
          \noindent Finally, hypothesis (H6)(ii), the fact that $\underset{n\rightarrow + \infty}{\lim} n \sqrt{h_n\,} = +\infty $ and the results \vspace*{-0.35cm}

          \begin{center}
          $h_n^{^{p-\frac{1}{\,2\,}}} B(2p,n,x)\ {\overrightarrow{\underset{n\rightarrow + \infty}}}\ \frac{1}{\,\sqrt{2p\,}\,\left(\sqrt{2\,\pi\,x\,}\right)^{^{2p-1}}\,}\ \ \text{and}\ \ E\!\left(f\!\left(G_{(2p,n,x)}\right)\right)\ {\overrightarrow{\underset{n\rightarrow + \infty}}}\ f(x),$\vspace*{-0.1cm} \end{center}
stated in Lemmas \ref{lema1} b) and \ref{lema2Hille}, complete the proof of a).

\noindent b) For $x=0$, the proof is similar to the previous one. It suffices to replace $\sqrt{h_n\,}$ by $h_n$, to use the fact that $B(2,n,0) = \frac{1}{\,2\,h_n\,}$ and to apply Lemma \ref{lema2Hille}.\vspace*{0.25cm}
\end{proof}

\begin{remark} \label{Vt,n.Dens} \textup{Analogously, the proof of the condition corresponding to (\ref{(2)}) in Theorem \ref{TeorTLCdensity} relies essentially on the inequalities\vspace*{-0.1cm}
$$ n \,E\!\left(\left|V_{t,n}(x)\right|^{^{2p}}\right)  \leqslant  C \frac{\,h_n^{^{\frac{\,p\,}{2}}}\,}{n^{^{p^{^{}}}}}\, B(2p,n,x)\, E\!\left(f\!\left(G_{(2p,n,x)}\right)\right)\,, \ \ \ x>0\,,$$
                        $$ n \,E\!\left(\left|V_{t,n}(0)\right|^{^{2p}}\right)  \leqslant  C \frac{\,h_n^{p}\,}{n^{^{{p-1}^{^{}}}}}\, B(2p,n,0)\, E\!\left(f\!\left(G_{(2p,n,0)}\right)\right)$$
\noindent and uses the same arguments as before.}

 \end{remark}


\begin{thebibliography}{99}
\bibitem{Benz:16}
\textsc{Benziadi, F., Laksaci, A., \textup{and} Tebboune, F.} (2016). Recursive kernel estimate
of the conditional quantile for functional ergodic data, \textit{Comm. Statist. Theory
Methods}, \textbf{45}, 3097--3113.

\bibitem{BouGhMes:11}
\textsc{Bouezmarni, T.; El Ghouch, A. \textup{and} Mesfioui, M.} (2011). Gamma kernel estimators for density and hazard rate of right-censored data, \textit{J. Probab. Stat.}, Vol. \textbf{2011} (Article ID 937574), 1--16.

\bibitem{BouRomb:08}
\textsc{Bouezmarni, T. \textup{and} Rombouts, J.V.K.} (2008). Nonparametric density and hazard function estimation for censored positive time series, \textit{J. Nonparametr. Stat.}, \textbf{20}, 627--643.

\bibitem{BouRomb:10a}
\textsc{Bouezmarni, T.} and \textsc{Rombouts, J. V.} (2010). Nonparametric density estimation
for multivariate bounded data. \textit{J. Statist. Plann. Inference},
\textbf{140}, 139--152.

\bibitem{BouRomb:10b}
\textsc{Bouezmarni, T. \textup{and} Rombouts, J.V.K.} (2010). Nonparametric density estimation for positive time series, \textit{Comput. Statist. Data Anal.}, \textbf{54}, 245--261.

\bibitem{Bouz:15}
\textsc{Bouzebda, S.} and \textsc{Didi, S.} (2015). Multivariate wavelet density and regression estimators
for stationary and ergodic discrete time processes: Asymptotic results.
\textit{Comm. Statist. Theory Methods}, (just-accepted).

\bibitem{BrChen:99}
\textsc{Brown, B. \textup{and} Chen, S.X.} (1999). Beta-bernstein smoothing for regression curves with compact support, \textit{Scand. J. Stat.}, \textbf{26}, 47--59.

\bibitem{ChaKhar:15}
\textsc{Chaouch, M. \textup{and} Khardani, S.} (2015). Kernel-smoothed conditional quantiles of randomly censored functional stationary ergodic data, \textit{J. Nonparametr. Stat.}, \textbf{27}, 65--87.

\bibitem{ChauLaiSen:12}
\textsc{Chaubey, Y.; La\"{\i}b, L. \textup{and} Sen, A.} (2010). Generalised kernel smoothing for nonnegative stationary ergodic processes, \textit{J. Nonparametr. Stat.}, \textbf{22}, 973--997.

\bibitem{Chen:99}
\textsc{Chen, S.X.} (1999). A beta kernel estimation for density functions, \textit{Comput. Statist. Data Anal.}, \textbf{31}, 131--145.

\bibitem{Chen:00}
\textsc{Chen, S.X.} (2000). Probability density function estimation using gamma kernels, \textit{Ann. Inst. Statist. Math.}, \textbf{52}, 471--480.


\bibitem{ChowTeicher:88}
\textsc{Chow, Y.S. \textup{and} Teicher, H.} (1988). \textit{Probability Theory}, 2nd Ed., Springer-Verlag, New York.

\bibitem{DaiSperl:10}
\textsc{Dai, J. \textup{and} Sperlich, S.A.} (2010). Simple and effective boundary correction for kernel densities and regression with an application to the world income and engel curve estimation. \textit{Comput. Statist. Data Anal.}, \textbf{54}, 2487--2497.

\bibitem{DelNogRosa:92}
\textsc{Delecroix, M.; Nogueira, M.E. \textup{and} Rosa, A.C.} (1992). Sur l'estimation de la densit\'{e} d'observations ergodiques, \textit{Statist. Anal. Donn\'{e}es}, \textbf{16}, 25--38.

\bibitem{DelRosa:96}
\textsc{Delecroix, M. \textup{and} Rosa, A.C.} (1996). Nonparametric estimation of a regression function and its derivatives under an ergodic hypothesis,
\textit{J. Nonparametr. Stat.}, \textbf{6}, 367--382.

\bibitem{Feller:70}
\textsc{Feller, W.} (1971). \textit{An Introduction to Probability Theory and its Applications}, Vol. 2, 2nd Ed., J. Wiley and Sons, New York.

\bibitem{GassMuller:79}
\textsc{Gasser, T. \textup{and} M\"{u}ller, H.G.} (1979). \textit{Kernel estimation of regression functions}. In ''Smoothing Techniques for Curve Estimation`` (T. Gasser and M. Rosenblatt Eds.), Lecture Notes in Mathematics, 757, 23--68, Springer, Berlin / Heidelberg.

\bibitem{GassMullerMammit:85}
\textsc{Gasser, T.; M\"{u}ller, H.G. \textup{and} Mammitzsch, V. (1985)}. Kernels for nonparametric curve estimation, \textit{J. Roy. Statist. Soc., Ser. B}, \textbf{47}, 238--252.

\bibitem{GeenWang:16}
\textsc{Geenens, G. \textup{and} Wang, C.} (2016). Local-likelihood transformation kernel density estimation for positive random variables, arXiv:1602.04862 [stat.ME].

\bibitem{GospHiru:07}
\textsc{Gospodinov, N.} and \textsc{Hirukawa, M.} (2007). Time series nonparametric regression
using asymmetric kernels with an application to estimation of scalar diffusion
processes. \textit{CIRJE Discussion Paper}, \textbf{F-573}, University of Tokyo.

\bibitem{Gyorfi:81}
\textsc{Gy\"{o}rfi, L.} (1981). Strong consistency density estimates from ergodic sample, \textit{J. Multivariate Anal}., \textbf{11}, 81--84.

\bibitem{HallHeide:80}
\textsc{Hall, P. \textup{and} Heide‬, C.} (1980). \textit{Martingale Limit Theory and its Applications}, Academic Press, New York.

\bibitem{HiruSak:14}
\textsc{Hirukawa, M.} and \textsc{Sakudo, M.} (2014). Nonnegative bias reduction methods for
density estimation using asymmetric kernels. \textit{Comput. Statist. Data
Anal.}, \textbf{75}, 112--123.

\bibitem{HuangChenWu:14}
\textsc{Huang, Y.; Chen, X. \textup{and} Wu, W.B.} (2014). Recursive nonparametric estimation for time series, \textit{IEEE Trans. Inform. Theory}, \textbf{60}, 1301--1312.

\bibitem{KarunAlberts:05}
\textsc{Karunamuni, R.J. \textup{and} Alberts, T.} (2005). On boundary correction in kernel density estimation, \textit{Stat. Methodol.}, \textbf{2}, 191–-212.

\bibitem{KoulSong:13}
\textsc{Koul, H. \textup{and} Song, W.} (2013). Large sample results for varying kernel regression estimates, \textit{J. Nonparametr. Stat.}, \textbf{25}, 829-–853.

\bibitem{LaibLouani:10}
\textsc{La\"{i}b, N. \textup{and} Louani, D.} (2010). Nonparametric kernel regression estimation for functional stationary ergodic data: asymptotic properties, \textit{J. Multivariate Anal.}, \textbf{101}, 2266--2281.

\bibitem{LaibLouani:11}
\textsc{La\"{i}b, N. \textup{and} Louani, D.} (2011). Rates of strong consistencies of the regression function estimator for functional stationary ergodic data, \textit{J. Statist. Plann. Inference}, \textbf{141}, 359-–372.

\bibitem{LaibOuldSaid:00}
\textsc{La\"{i}b, N. \textup{and} Ould-Sa\"{i}d, E.} (2000). A robust nonparametric estimation of the autoregression function under an ergodic hypothesis, \textit{Canad. J. Statist.}, \textbf{2}, 817--828.

\bibitem{LingLiu:16}
\textsc{Ling, N.} and \textsc{Liu, Y.} (2016). The kernel regression estimation for randomly
censored functional stationary ergodic data. \textit{Communications in Statistics-
Theory and Methods}, (just-accepted).

\bibitem{LingLiuVieu:16}
\textsc{Ling, N., Liu, Y.,} and \textsc{Vieu, P.} (2016). Conditional mode estimation for functional
stationary ergodic data with responses missing at random. \textit{Statistics}, 1--23.

\bibitem{MaleSchi:14}
\textsc{Malec, P. \textup{and} Schienle, M.} (2014). Nonparametric kernel density estimation near the boundary, \textit{Comput. Statist. Data Anal.}, \textbf{72}, 57--76.

\bibitem{Markovich:15}
\textsc{Markovich, L.A.} (2015). Gamma kernel estimation of the density derivative on the positive semi-axis by dependent data, to appear in \textit{RevStat}.

\bibitem{Rice:84}
\textsc{Rice, J.} (1984). Boundary modification for kernel regression,
\textit{Comm. Statist. Theory Methods}, \textbf{13}, 893--900.

\bibitem{Scaillet:04}
\textsc{Scaillet, O.} (2004). Density estimation using inverse gaussian and reciprocal inverse gaussian kernels, \textit{J. Nonparametr. Stat.}, \textbf{16}, 217--226.

\bibitem{Schuster:85}
\textsc{Schuster, E.} (1985). Incorporating support constraints into nonparametric estimators of densities. \textit{Comm. Statist. Theory Methods}, \textbf{14}, 1123–-1136.

\bibitem{ShiSong:15}
\textsc{Shi, J. \textup{and} Song, W.} (2016). Asymptotic results in gamma kernel regression. \textit{Comm. Statist. Theory  Methods}, \textbf{45}, 3489–-3509.

\bibitem{Simonoff:98}
\textsc{Simonoff, J.} (1998). \textit{Smoothing Methods in Statistics}, Springer Verlag (2nd Ed.), New York.

\bibitem{VanRyzin:69}
\textsc{Van Ryzin, J.} (1969). On strong consistency of density estimates, \textit{Ann. Math. Stat.}, \textbf{40}, 1765--1772.

\bibitem{Wu:03}
\textsc{Wu, W.B.} (2003). Nonparametric estimation for stationary processes, Technical Report 536, University of Chicago.

\bibitem{YakGyorfKiefferMorvai:99}
\textsc{Yakowitz, S.; Gy\"{o}rfi, L.; Kieffer, J. \textup{and} Morvai, G.} (1999). Strongly consistent nonparametric forecasting and regression for stationary ergodic sequences. \textit{J. Multivariate Anal.}, \textbf{71}, 24–-41.

\end{thebibliography}
\end{document}